\documentclass[12pt, reqno]{amsart}
\usepackage{amsmath, amsthm, amscd, amsfonts, amssymb, graphicx, color}
\usepackage[bookmarksnumbered, colorlinks, plainpages]{hyperref}

\newtheorem{theorem}{Theorem}[section]
\newtheorem{lemma}[theorem]{Lemma}

\theoremstyle{definition}
\newtheorem{definition}[theorem]{Definition}
\newtheorem{example}[theorem]{Example}

\theoremstyle{remark}

\numberwithin{equation}{section}

\begin{document}
\setcounter{page}{1}

\title[Frames on Krein Space]{Frames on Krein Spaces}

\author[S. Karmakar, Sk M. Hossein]{Shibashis Karmakar$^1$ and Sk. Monowar Hossein$^2$$^{*}$}

\address{$^{1}$ Department of Mathematics, Jadavpur University, Jadavpur-32, West Bengal, India.}
\email{\textcolor[rgb]{0.00,0.00,0.84}{shibashiskarmakar@gmail.com}}

\address{$^{2}$ Department of Mathematics, Aliah University, Sector-V, Kolkata-91, West Bengal, India.}
\email{\textcolor[rgb]{0.00,0.00,0.84}{sami$\_$milu@yahoo.co.uk}}




\date{Received: xxxxxx; Revised: yyyyyy; Accepted: zzzzzz.
\newline \indent $^{*}$ Corresponding author}

\begin{abstract}
In this article we define frame for a Krein space $\textbf{\textit{K}}$ with a $J$-orthonormal basis and extend the 
notion of frame sequence and frame potential analogous to Hilbert spaces.We show that every frame is a sum of three orthonormal
bases of a Krein space. We also find relation between frame and orthogonal projections on Krein space.
\end{abstract} \maketitle

\section{Introduction and preliminaries}

Let $\textbf{\emph{H}}$ be a real or complex Hilbert space. A sequence $\{f_n:n\in\mathbb{N}\}$ is said to be a frame for Hilbert space $\textbf{\emph{H}}$ if there exists positive reals $A$ and $B$ s.t. $A\|f\|^2\leq\sum_{n\in\mathbb{N}}|\langle{f},f_n\rangle|^2\leq{B\|f\|^2}$ for all $f\in{\textbf{\emph{H}}}$. Frame for Hilbert spaces was defined by Duffin and Schaeffer \cite{ds} in 1952 to study some problems in nonharmonic Fourier series. But Daubechies et. al.\cite{dgm} published a landmark paper in 1986 in this direction, working on wavelets and signal processing. After the work, the theory of frames begun to be more widely studied. Powerful tools from Operator theory and Banach space theory were being introduced to study Frames in Hilbert spaces and it produced some deep results in Frame theory. 
Many researchers studied frame in different aspects and applications \cite{oc,pgc,pcs,phl}. Krein space theory \cite{jb} is rich 
in application among the many areas of Mathematics. In the year 2011, J. I. Giribet et. al.\cite{gmmm} introduced frame 
in Krein spaces which is known as $J$-frames. But according to their definition, every spanning set of finite dimensional krein space $\textbf{\textit{K}}$ need not be a $J$-frame for $\textbf{\textit{K}}$ 
which is a drawback, as in finite-dimensional Hilbert spaces, each spanning set is a frame. 
So it is desirable to introduce a new definition which will solve the above shortcomings of 
the definition of \cite{gmmm}. Recently K. Esmeral et. al. (\cite{koe}) defined frames in more general 
setting in Krein spaces. According to their definition a countable 
sequence $\{f_n\}_{n\in{\mathbb{N}}}\textmd{ in }{\textbf{\textit{K}}}$ is called a frame for $\textbf{\textit{K}}$, 
if there exist constants $0<A\leq{B}<\infty$ such that $A\|f\|_J^2\leq\sum_{n\in\mathbb{N}}|[{f},f_n]|^2\leq{B\|f\|_J^2}$ 
for all $f\in{\textbf{\textit{K}}}$. But the definition involves fundamental symmetry of the Krein 
space $\textbf{\textit{K}}$ which is not unique. In this paper we have defined frame in Krein spaces which is more suitable 
in application purposes i.e. we can study all the relevant properties of Frames, like Hilbert spaces.
\section{Motivation and Basic DefinitionS}
Let $\textbf{\textit{K}}$ be a Krein space and $\{e_{j}\}_{j=1}^{\infty}$ be a 
countable $J$-orthonormal basis for $\textbf{\textit{K}}$. 
Let $\textbf{\textit{K}}^{+}=\{\overline{span\{e_i\}}:\langle~e_i,e_i\rangle=1\}$ 
and $\textbf{\textit{K}}^{-}=\{\overline{span\{e_i\}}:\langle~e_i,e_i\rangle=-1\}$. 
Then $\textbf{\textit{K}}^{+}$ is a maximal uniformly $J$-positive subspace i.e. a Hilbert space 
and $\textbf{\textit{K}}^{-}$ is a maximal uniformly $J$-negative subspace i.e. an anti-space of a Hilbert space. 
So we have $\textbf{\textit{K}}=\textbf{\textit{K}}^+(\dot{+})\textbf{\textit{K}}^-$ i.e. 
a fundamental decomposition of $\textbf{\textit{K}}$.
Before we define frame for Krein space, we will look the following definition for our motivation.
\begin{example}
 Consider the Krein space $\mathbb{R}^3$ over $\mathbb{R}$, where the inner product is given 
by $\langle~e_1,e_1\rangle=1,\langle~e_2,e_2\rangle=1,\langle~e_3,e_3\rangle=-1,\langle~e_1,e_2\rangle=0,\langle~e_2,e_3\rangle=0,\langle~e_3,e_1\rangle=0$ 
(where $\{e_1,e_2,e_3\}$ be the standard basis for $\mathbb{R}^3$).\\ Now $F=\{(1,0,\frac{1}{\sqrt{2}}),(0,1,\frac{1}{\sqrt{2}}),(0,0,1)\}$ 
is a spanning set of $\mathbb{R}^3$. Hence it is a frame for the associated Hilbert space $\mathbb{R}^3$. But F is not a $J$-frame. 
\end{example}
\begin{example}
 Consider the vector space $\ell^2(\mathbb{N})$ over $\mathbb{R}$. Let $x=(\alpha_1,\alpha_2,\ldots),\\
~y=(\beta_1,\beta_2,\ldots)~\in\ell^2$,\\
we define $\langle{x},y\rangle=\alpha_1\beta_1-\alpha_2\beta_2+\alpha_3\beta_3-\alpha_4\beta_4\ldots$\\
Then $(\ell^2,\langle\cdot\rangle)$ is a Krein space with fundamental symmetry $J_1:\ell^2\rightarrow\ell^2$ defined 
by $J_1(\alpha_1,\alpha_2,\ldots)=(\alpha_1,-\alpha_2,\ldots)$, but we can define another fundamental symmetry $J_2$ 
on $(\ell^2,\langle\cdot\rangle)$ s.t. $J_2:\ell^2\rightarrow\ell^2$ defined 
by $J_1(\alpha_1,\alpha_2,\alpha_3,\alpha_4,\ldots)=(2\alpha_1-\sqrt{3}\alpha_2,\sqrt{3}\alpha_1-2\alpha_2,\alpha_3,-\alpha_4,\ldots)$.
Now according to the definition given by K. Esmeral et. al. \cite{koe}, the sequence $ \{e_n:n\in{\mathbb{N}}\} $ is a 
frame for the triple $(\ell^2,\langle\cdot\rangle,J_1)$ with frame bound $A=B=1$, but if we consider the 
triple $(\ell^2,\langle\cdot\rangle,J_2)$, then the sequence $\{e_n:n\in{\mathbb{N}}\}$ is not a frame with frame 
bound $A=B=1$. So the definition given in \cite{koe} takes away some important geometries of Parseval frame. \\

\end{example}
Motivating by the above example, we define frame in the following manner.
\begin{definition}
 A sequence $\{f_{n}:n\in\mathbb{N}\}\textmd{ in }{\textbf{\textit{K}}}$ is said to be a frame for $\textbf{\textit{K}}$ if $\{f_{n}^+:n\in\mathbb{N}\}$ and $\{f_{n}^-:n\in\mathbb{N}\}$ are frames for $\textbf{\textit{K}}^+$ and $\textbf{\textit{K}}^-$ respectively.
\end{definition}

  In mathematical terms A sequence $\{f_{n}:n\in\mathbb{N}\}\textmd{ in }{\textbf{\textit{K}}}$ is said to be a frame for $\textbf{\textit{K}}$ if the following inequalities are satisfied
\begin{equation}
A_1\langle x^{+},x^{+}\rangle~\leq~\sum_{n\in\mathbb{N}}|\langle x^{+},f_{n}^{+}\rangle|^2~\leq~B_1\langle x^{+},x^{+}\rangle
\end{equation}
\begin{equation}
A_2\langle x^{-},x^{-}\rangle~\leq~\sum_{n\in\mathbb{N}}|\langle x^{-},f_{n}^{-}\rangle|^2~\leq~B_2\langle x^{-},x^{-}\rangle
\end{equation}
where $\langle{\cdot}\rangle$ is the inner product of the Krein space and $B_2\leq~A_2<0<A_1\leq~B_1$ are real numbers. 
Also $x^+,f_n^+\in{\textbf{\textit{K}}}^+~\textmd{and}~x^-,f_n^-\in{\textbf{\textit{K}}}^-$ respectively $\forall~n\in\mathbb{N}$.

The frame theory of an anti-Hilbert space is same as the frame theory of Hilbert space except a change in sign in the inner product of the 
anti-Hilbert space. So in definite inner product space, when we say that a sequence is a frame for that space, then we mean that the sequence 
is a frame for that space in Hilbert space sense.
\begin{definition}
A sequence $\{f_{n}:n\in\mathbb{N}\}\textmd{ in }{\textbf{\textit{K}}}$ is called a tight frame if $B_2=A_2=C(\textmd{say})$ and 
$B_1=A_1=D(\textmd{say})$ with $C=-D$ and moreover if $D=1$, then the frame is called a Parseval frame.
\end{definition}
Before we go further, we give the following example to show that one sequence may be frame in the sense of Krein space 
but not in the sense of Hilbert space.
\begin{example}
 Consider the vector space $\mathbb{R}^2$ over $\mathbb{R}$. Let $\{e_1,e_2\}$ be the standard basis for $\mathbb{R}^2$. 
Define $\langle~e_1,e_1\rangle=1,\langle~e_2,e_2\rangle=-1,~\langle~e_1,e_2\rangle=0$. Then $\mathbb{R}^2$ is a Krein space under the above 
defined inner product. Here $\{(1,1)\}$ is a frame for the Krein space $\mathbb{R}^2$. But $\{(1,1)\}$ is not a frame for the associated Hilbert
 space, because $\{(1,1)\}$ is not a spanning set of $\mathbb{R}^2$. 
\end{example}

\begin{theorem}
Every spanning set in a finite-dimensional Krein space is a frame for that space.
\end{theorem}
\begin{proof}
Let $ {\textbf{\textit{K}}}$ be a finite dimensional Kreain space and $\{f_{k}\}_{k=1}^{m}$ 
be a spanning set in $\textbf{\textit{K}}$.We want to show that $\{f_{k}\}_{k=1}^{m}$ is a 
frame for $\textbf{\textit{K}}$.\\
Let $\{e_i\}_{i=1}^{n}$ be a J-orthonormal basis for $\textbf{\textit{K}}$, where n is the dimension of $\textbf{\textit{K}}$. Then by the 
above J-orthonormal basis we have $\textbf{\textit{K}}=\textbf{\textit{K}}^+(\dot{+})\textbf{\textit{K}}^-$ 
where $\textbf{\textit{K}}^{+}$ is a Hilbert space and $\textbf{\textit{K}}^{-}$ is a anti-space of a Hilbert space.
Let $B_{1}=\|f_{1}^{+}\|_J^{2}+\ldots+\|f_{m}^{+}\|_J^{2}$ and $B_{2}=-(\|f_{1}^{-}\|_J^{2}+\ldots+\|f_{m}^{-}\|_J^{2})$. 
Then the upper frame conditions are satisfied.\\
Now consider the function $\phi_1:\textbf{\textit{K}}^+\rightarrow{\mathbb{R}}$ defined by
\begin{equation*}
\phi_1(f^+)=\sum_{k=1}^{m}|\langle f^{+},f_{k}^{+}\rangle|^{2},
\end{equation*}
where $f\in~\textbf{\textit{K}},~f^{+}\in~\textbf{\textit{K}}^{+}~\textmd{and}~f^{-}\in~\textbf{\textit{K}}^{-}$. 
Obviously, $\phi_1$ is a continuous function.\\
Consider the set $S_{1}=\{\sum_{k=1}^{m}|\langle f^{+},f_{k}^{+}\rangle|^{2};~f^{+}\in~\textbf{\textit{K}}^{+},\textmd{and}~\|f^{+}\|=1\}$.\\
Let
\begin{equation*}
A_{1}:=\sum_{k=1}^{m}|\langle g^{+},f_{k}^{+}\rangle|^{2}=inf~S_{1}~\textmd{where}~g^{+}\in{\textbf{\textit{K}}}^{+}~\textmd{and}~\|g^{+}\|=1.
\end{equation*}
It is clear that $A_{1}>0$ and, we have
$~\sum_{k=1}^{m}|\langle f^{+},f_{k}^{+}\rangle|^{2}\geq~A_{1}\|f^{+}\|^{2},~\forall f^{+} \in \textbf{\textit{K}}.$

Again, as $\textbf{\textit{K}}^{-}$ is the anti-space of a Hilbert space, so the inner product on $\textbf{\textit{K}}^{-}$ denoted 
by $\langle{\cdot}\rangle^\prime~\textmd{and defined as}~\langle{x,y}\rangle^\prime=-\langle{x,y}\rangle$, 
where $x,y\in{\textbf{\textit{K}}^{-}}$, turns $\textbf{\textit{K}}^{-}$ into a Hilbert space. 
Since the topology induced by the indefinite inner product on the Krein space is a majorant topology,
 so the intrinsic norm will be invarianted under the above transformation.

Now set $S_{2}=\{\sum_{k=1}^{m}|\langle f^{-},f_{k}^{-}\rangle|^{2}; ~f^{-}\in~\textbf{\textit{K}}^{-},\textmd{and}~\langle~f^{-},f^{-}\rangle=-1\}$\\
Then as above we have $A_2<0$ such that
\begin{equation*}
\sum_{k=1}^{m}|\langle f^{-},f_{k}^{-}\rangle|^{2}\geq~A_{2}\langle~f^{-},f^{-}\rangle~\forall f^{-}\in~\textbf{\textit{K}}^{-}.
\end{equation*}
So the lower frame conditions are satisfied. Hence the proof.
\end{proof}
\subsection{Representation of elements w.r.t. given frame.}
Let $\{f_n\}_{n=1}^{\infty}$ be a frame for the Krein space $\textbf{\textit{K}}$ which may be real or complex. With respect to 
a given $J$-orthonormal basis we have $\textbf{\textit{K}}=\textbf{\textit{K}}^+(\dot{+})\textbf{\textit{K}}^-$ 
where $\textbf{\textit{K}}^+$ and $\textbf{\textit{K}}^-$ are Hilbert space and an anti-space of a Hilbert space respectively. 
So $\{f_n^+\}_{n=1}^{\infty}$ is a frame for the Hilbert space $\textbf{\textit{K}}^+$ and $\{f_n^-\}_{n=1}^{\infty}$ is a frame 
for the Hilbert space $(\textbf{\textit{K}}^-,\langle~.~\rangle^\prime)$ (i.e. w.r.t. the new inner product). 
Let $S_1$ and $S_2$ be the corresponding frame operators, then we have
\begin{equation*}
S_1(f^+)=\sum_{n=1}^{\infty}\langle f^{+},f_{n}^{+}\rangle{f_{n}^{+}},~~\forall~f^+\in{\textbf{\textit{K}}^+}
\end{equation*}
\begin{equation*}
\textmd{and}~~S_2(f^-)=\sum_{n=1}^{\infty}\langle f^{-},f_{n}^{-}\rangle^\prime{f_{n}^{-}},~~\forall~f^-\in{\textbf{\textit{K}}^-}.
\end{equation*}
So for a given $f\in{\textbf{\textit{K}}}$ and $\{f_n\}_{n=1}^{\infty}$ which is a frame for $\textbf{\textit{K}}$ we have a representation for $f$   i.e.
\begin{equation*}
f=\sum_{n=1}^{\infty}\langle f^{+},S_1^{-1}f_{n}^{+}\rangle{f_{n}^{+}}-\sum_{n=1}^{\infty}\langle f^{-},S_2^{-1}f_{n}^{-}\rangle{f_{n}^{-}}.
\end{equation*}
\\
\section{Main results}

\subsection{General Properties}
In Hilbert space frame theory, every frame is a sum of three orthonormal bases \cite{pgi}. We show a similar result in Krein space 
frame theory.
\begin{theorem}
Assume that $\textbf{\textit{K}}$ is a complex Krein space and that $\{f_k\}_{k=1}^{\infty}$ is a frame for $\textbf{\textit{K}}$ 
with pre-frame operators $T^+$ and $T^-$. Then, for every $\epsilon\in(0,1)$ there exist three $J$-orthonormal 
bases $\{e_k^1\}_{k=1}^\infty$, $\{e_k^2\}_{k=1}^\infty$ and $\{e_k^3\}_{k=1}^\infty$ for $\textbf{\textit{K}}$ s.t.
\begin{equation*}
f=\frac{1}{1-\epsilon}\{\|T^+\|_J({e_k^+}^1+{e_k^+}^2+{e_k^+}^3)+\|T^-\|_J({e_m^-}^1+{e_m^-}^2+{e_m^-}^3)\}~~\forall{~k,m}\in{\mathbb{N}}
\end{equation*}
\end{theorem}
\begin{proof}
Let $\{e_k\}_{k=1}^\infty$ denote a $J$-orthonormal basis for $\textbf{\textit{K}}$ and let $\{\delta_k\}_{k=1}^\infty$ be the 
cannonical orthonormal basis for $\ell^2(\mathbb{N})$. Composing the pre-frame operator for $\{f_k^+\}_{k=1}^{\infty}$ with the 
isometric isomorphism from $\textbf{\textit{K}}^+$ to $\ell^2(\mathbb{N})$ which maps $e_k^+$ to $\delta_k$ 
(here $e_k^+=e_k\textmd{ if }\langle{e_k,e_k}\rangle=1,\textmd{ and }e_k^-:=e_k, \textmd{~if~} \langle{e_k,e_k}\rangle=-1,)$, 
we obtain a bounded linear operator of $\textbf{\textit{K}}^+$ onto $\textbf{\textit{K}}^+$ (we consider $\textbf{\textit{K}}^+$ as 
a Hilbert space i.e. w.r.t. the intrinsic topology), which maps $e_k^+$ to $f_k^+$, we denote it by $T^+$. Now for a given 
$\epsilon\in(0,1)$, consider the operator
\begin{equation*}
U^+:\textbf{\textit{K}}^+\rightarrow{\textbf{\textit{K}}^+},~\textmd{where}~U^+=\frac{1}{2}I+\frac{1-\epsilon}{2}\frac{T^+}{\|T^+\|_J}.
\end{equation*}
 Then $U^+$ is invertible and we can write $U^+=V^+P^+$, where $V^+$ is unitary and $P^+$ is a positive operator.\\
Observe that,
$\|P^+\|_J<1$.\\
The expression for $T^+$ is
\begin{equation*}
T^+=\frac{\|T^+\|_J}{1-\epsilon}(V^+W^{+}+V^+{W^{+}}^{\ast}-I)
\end{equation*}
It follows from here that,
\begin{equation}
f_k^+=\frac{\|T^+\|_J}{1-\epsilon}({e_k^+}^1+{e_k^+}^2+{e_k^+}^3).
\end{equation}
Similarly composing the pre-frame operator for $\{f_k^-\}_{k=1}^{\infty}$ with the isometric isomorphism 
from $\textbf{\textit{K}}^-$ to $\ell^2(\mathbb{N})$ which maps $e_m^-$ to $\delta_m$ (here $e_k^-= e_k:\langle{e_k,e_k}\rangle=-1$), 
we obtain a bounded linear operator of $\textbf{\textit{K}}^-$ onto $\textbf{\textit{K}}^-$ 
(we consider $(\textbf{\textit{K}}^-,{\langle{\cdot}\rangle}^\prime)$ as a Hilbert space i.e. w.r.t. the intrinsic topology), 
which maps $e_k^-$ to $f_k^-$, and we denote it by $T^-$. For a given $\epsilon\in(0,1)$, consider the operator
\begin{equation*}
U^-:\textbf{\textit{K}}^-\rightarrow{\textbf{\textit{K}}^-},~~~U^-=\frac{1}{2}I+\frac{1-\epsilon}{2}\frac{T^-}{\|T^-\|_J}.
\end{equation*}
Since $\|I-U^-\|_J<1$, so $U^-$ is invertible and we have $U^-=V^-P^-$, where $V^-$ is unitary when $K^-$ is considered as a Hilbert space and $P^-$ is a positive operator.\\
Following the same arguments as above, we have
\begin{equation}
f_m^-=\frac{\|T^+\|_J}{1-\epsilon}({e_m^-}^1+{e_m^-}^2+{e_m^-}^3).
\end{equation}
So by (3.1) and (3.2) we have our desire result. 
\end{proof}
Orthogonal projections play a special role in many contexts. We state a relationship between frames and orthogonal projections \cite{oc}.
\begin{theorem}
Let $\textbf{\textit{K}}$ be a Krein space with fundamental symmetry $J$, and let $P$ be an orthogonal projection commuting with $J$.\\
If $\{f_n\}_{n\in{\mathbb{N}}}$ is a frame for $\textbf{\textit{K}}$ with frame bounds $B_2\leq{A_2}<0<A_1\leq{B_1}$, then $\{Pf_n\}_{n\in{\mathbb{N}}}$ is a frame for $P\textbf{\textit{K}}$ and $\{(I-P)f_n\}_{n\in{\mathbb{N}}}$ is a frame for $(I-P)\textbf{\textit{K}}$, both admitting the same frame bounds.

Conversely, if $\{f_n\}_{n\in{\mathbb{N}}}$ is a frame for $P\textbf{\textit{K}}$ and $\{g_n\}_{n\in{\mathbb{N}}}$ is a frame for $(I-P)\textbf{\textit{K}}$, both with frame bounds $B_2\leq{A_2}<0<A_1\leq{B_1}$, then $\{f_n\}_{n\in{\mathbb{N}}}\bigcup{\{g_n\}_{n\in{\mathbb{N}}}}$ is a frame for $\textbf{\textit{K}}$ admitting the same frame bounds.
\end{theorem}
\begin{proof}
Since $P$ commutes with $J$, the subspaces $P\textbf{\textit{K}}$ and $(I-P)\textbf{\textit{K}}$ are Krein spaces with fundamental 
symmetry $PJ$ and $(I-P)J$, respectively. Since $\{f_n\}_{n\in{\mathbb{N}}}$ is a frame for $\textbf{\textit{K}}$ with frame 
bounds $B_2\leq{A_2}<0<A_1\leq{B_1}$ and with fundamental symmetry $J$,\\
We have,\\
\begin{equation*}
\sum_{n\in\mathbb{N}}|\langle{P}x^{+},f_{n}^{+}\rangle|^2=\sum_{n\in\mathbb{N}}|\langle{x}^{+},Pf_{n}^{+}\rangle|^2
\end{equation*}
and
\begin{equation*}
\sum_{n\in\mathbb{N}}|\langle{P}x^{-},f_{n}^{-}\rangle|^2=\sum_{n\in\mathbb{N}}|\langle{x}^{-},Pf_{n}^{-}\rangle|^2
\end{equation*}
So from the above two equations we conclude that $\{Pf_n\}_{n\in{\mathbb{N}}}$ is a frame for $P\textbf{\textit{K}}$.\\
The same remains true for $P$ replaced by $(I-P)$. Hence the result.

For the converse, let $\{f_n\}_{n\in{\mathbb{N}}}$ and $\{g_n\}_{n\in{\mathbb{N}}}$ be two frames satisfying the assumptions stated in the proposition. For $k\in{\textbf{\textit{K}}}$, set $k_1:=Pk~\textmd{ and }k_2:=(I-P)k$. Note that,
\begin{equation*}
\langle{k^+},f_n^+\rangle=\langle{k^+},Pf_n^+\rangle=\langle{Pk^+},f_n^+\rangle=\langle{k_1^+},f_n^+\rangle
\end{equation*}
Similarly
\begin{equation*}
\langle{k^-},f_n^-\rangle=\langle{k_1^-},f_n^-\rangle,~\langle{k^+},g_n^+\rangle=
\langle{k_2^+},g_n^+\rangle,~\langle{k^-},g_n^-\rangle=\langle{k_2^-},g_n^-\rangle\\
\end{equation*}
From $PJ=JP$, it follows that,
\begin{equation*}
\langle{k_1^+},k_1^+\rangle+\langle{k_2^+},k_2^+\rangle=\langle{k^+},k^+\rangle~\textmd{and }
\langle{k_1^-},k_1^-\rangle+\langle{k_2^-},k_2^-\rangle=\langle{k^-},k^-\rangle
\end{equation*}
Now from the above equations we have our desired result. 
\end{proof}

\subsection{Frame sequences and their properties}
Frame sequence in a Hilbert space $\textbf{\emph{H}}$ is a well known concept. S.K.Kaushik et. al.\cite{sgv} published a paper in the year 2008 to show some important properties of frame sequences. We will show that the results they found also holds for frames in Krein Spaces.\\
Now let $\textbf{\textit{K}}$ be a Krein space.
\begin{definition}
A sequence $\{f_n:n\in{\mathbb{N}}\}\textmd{ in }{\textbf{\textit{K}}}$ is said to be a frame sequence if $\{f_n^+:n\in{\mathbb{N}}\}$ is a frame for $[f_n^+]$ and $\{f_n^-:n\in{\mathbb{N}}\}$ is a frame for the anti-Hilbert space $[f_n^-]$ respectively. ( where $[f_n^+]$ and $[f_n^-]$ denotes the closed linear span of $\{f_n^+\}$ and $\{f_n^-\}$ respectively w.r.t. the intrinsic topology )
\end{definition}
\begin{definition}
A frame $\{f_n\}\textmd{ in }{\textbf{\textit{K}}}$ is called exact if removal of an arbitrary $f_n$ renders the collection $\{f_n\}$ no longer a frame for  $\textbf{\textit{K}}$.
\end{definition}
\begin{definition}
A frame $\{f_n\}\textmd{ in }{\textbf{\textit{K}}}$ is called near exact if it can be made exact by removing finitely many elements from it.\\
Also a near exact frame is called proper if it is not exact.
\end{definition}
\begin{theorem}
Let $\{f_n\}$ be any frame of $\textbf{\textit{K}}$ and let $\{m_k\}$ and $\{n_k\}$ be two infinite increasing sequence of $\mathbb{N}$
 with $\{m_k\}\cup\{n_k\}=\mathbb{N}$. Also let $\{f_{m_{k}}\}$ be frame for $\textbf{\textit{K}}$. 
Then $\{f_{n_{k}}\}$ is a frame for $\textbf{\textit{K}}$ if there exist bounded linear operators $S_1$ and $S_2$ s.t.,
\begin{equation*}
S_1:\ell^2(\mathbb{N})\rightarrow\ell^2(\mathbb{N})~\textmd{ defined by }S_1\{\langle~f_{n_{k}}^+,f^+\rangle\}=\{\langle~f_{m_{k}}^+,f^+\rangle\},
\end{equation*}
and
\begin{equation*}
S_2:\ell^2(\mathbb{N})\rightarrow\ell^2(\mathbb{N})~\textmd{ defined by }S_2\{\langle~f_{n_{k}}^-,f^-\rangle\}=\{\langle~f_{m_{k}}^-,f^-\rangle\}.
\end{equation*}
The converse is true only if both $\{f_{m_{k}}\},~\{f_{m_{k}}\}$ are Parseval frames.
\end{theorem}
\begin{proof}
$(\Rightarrow)$ Since $\{f_{m_{k}}\}$ is a frame for $\textbf{\textit{K}}$, let $A_1$ and $A_2$ be lower bounds of the frames 
$\{f_{m_{k}}^+\}$ and $\{f_{m_{k}}^-\}$ respectively. Now we have,
\begin{equation*}
\sum|\langle~f_{n_{k}}^+,f^+\rangle|^2\geq~\frac{\sum|\langle~f_{m_{k}}^+,f^+\rangle|^2}{\|S_1\|}\geq~\frac{A_1}{\|S_1\|}\|f^+\|^2.
\end{equation*}
Similarly we also have,
\begin{equation*}
\sum|\langle~f_{n_{k}}^-,f^-\rangle|^2\geq~\frac{\sum|\langle~f_{m_{k}}^-,f^-\rangle|^2}{\|S_2\|}\geq~\frac{A_2}{\|S_2\|}\|f^-\|^2.
\end{equation*}
Hence $\{f_{n_{k}}\}$ is also a frame for $\textbf{\textit{K}}$.\\
$(\Leftarrow)$ Conversely, let $\{f_{n_{k}}\}$ be an Parseval frame for $\textbf{\textit{K}}$. Then let the synthesis and analysis operators for $\textbf{\textit{K}}^+$ be $T_1^+$ and $(T^{\ast}_1)^+$ respectively and the same for $\textbf{\textit{K}}^-$ be $T_1^-$ and $(T^{\ast}_1)^-$.
Also, since $\{f_{m_{k}}\}$ is a frame for $\textbf{\textit{K}}$, there exist operators $T_2^+$ and $(T^{\ast}_2)^+$ for $\textbf{\textit{K}}^+$ and there exists operators $T_2^-$ and $(T^{\ast}_2)^-$ for $\textbf{\textit{K}}^-$.
Then $S_1=(T^{\ast}_2)^+T_1^+:\ell^2(\mathbb{N})\rightarrow\ell^2(\mathbb{N})$ is a bounded linear operator such that
$~S_1\{\langle~f_{n_{k}}^+,f^+\rangle\}=\{\langle~f_{m_{k}}^+,f^+\rangle\}$, and $S_2=(T^{\ast}_2)^-T_1^-:\ell^2(\mathbb{N})\rightarrow\ell^2(\mathbb{N})$ is a bounded linear operator such that
$~S_2\{\langle~f_{n_{k}}^-,f^-\rangle\}=\{\langle~f_{m_{k}}^-,f^-\rangle\}$. 
\end{proof}
\begin{theorem}
Let $\{f_n\}$ be any frame for $\textbf{\textit{K}}$ and let $\{m_k\}$, $\{n_k\}$ be two infinite increasing sequences with $\{m_k\}\cup\{n_k\}=\mathbb{N}$. Let $\textbf{\textit{K}}^{\prime}=[f_{m_k}]\cap[f_{n_k}]$.\\
If $\textbf{\textit{K}}^{\prime}$ is a finite dimensional space, then $\{f_{m_{k}}\}$ and $\{f_{n_{k}}\}$ are frame sequences for $\textbf{\textit{K}}$.
\end{theorem}
\begin{proof}
Let $\{l_k\}$ be any finite subsequence of $\{n_k\}$ such that $\textbf{\textit{K}}^{\prime}=[f_{l_{k}}]$.
 Let $\{f_{l_{k}}\}$ be a frame for $\textbf{\textit{K}}^\prime$ and $B_2^\prime<A_2^\prime<0<A_1^\prime<B_1^\prime$ be the 
bounds of the frame $\{f_{l_{k}}\}$. Consider $\{f_{n_{k}}\}$. Let $f\in~[f_{n_k}]$ be any element. 
Now, if $f\perp~\textbf{\textit{K}}^\prime$, then\\
$\sum|\langle~f^+,f_n^+\rangle|^2=\sum|\langle~f^+,f_{n_k}^+\rangle|^2\geq~A_1\|f^+\|^2$ (here $A_1>0$)\\
$\sum|\langle~f^-,f_n^-\rangle|^2=\sum|\langle~f^-,f_{n_k}^-\rangle|^2\geq~A_2\langle~f^-,f^-\rangle$ (here $A_2<0$)\\
Also, if $f\in~\textbf{\textit{K}}^\prime$ then\\
$\sum|\langle~f^+,f_{n_k}^+\rangle|^2\geq~\sum|\langle~f^+,f_{l_k}^+\rangle|^2\geq~A_1^\prime\|f^+\|^2$\\
$\sum|\langle~f^-,f_{n_k}^-\rangle|^2\geq~\sum|\langle~f^-,f_{l_k}^-\rangle|^2\geq~A_2^\prime\langle~f^-,f^-\rangle$\\
Otherwise, we have\\
$f^+=\sum\alpha_kf_{n_k}^+=(f^+)^\prime+(f^+)^{\prime\prime}$, and
$f^-=\sum\alpha_kf_{n_k}^-=(f^-)^\prime+(f^-)^{\prime\prime}$,\\
here $i\in~\{n_k\}-\{l_k\},~j\in~\{l_k\}$
and $f^\prime~\perp~\textbf{\textit{K}}^\prime$ and $f^{\prime\prime}\in~\textbf{\textit{K}}^\prime$.\\
Thus,\\
$\sum|\langle~f^+,f_{n_k}^+\rangle|^2=\sum|\langle~f^+,f_i^+\rangle|^2+\sum|\langle~f^+,f_j^+\rangle|^2,~i\in\{n_k\}-\{l_k\},~j\in\{l_k\}$\\
$=\sum|\langle~(f^+)^\prime,f_i^+\rangle|^2+\sum|\langle~(f^+)^{\prime\prime},f_j^+\rangle|^2$\\
$\geq~A_1\|(f^+)^\prime\|^2+A_1^\prime\|(f^+)^{\prime\prime}\|^2$\\
$\geq~min\{\frac{A_1}{2},\frac{A_1^\prime}{2}\}\|f^+\|^2$\\
Similarly we have,\\
$\sum|\langle~f^-,f_{n_k}^-\rangle|^2\geq~min\{\frac{A_2}{2},\frac{A_2^\prime}{2}\}\langle~f^-,f^-\rangle.$\\
Hence $\{f_{n_k}\}$ is a frame sequence. Similarly we can show $\{f_{m_k}\}$ is also a frame sequence.\\
\end{proof}
\begin{theorem}
Let $\{f_n\}$ be a frame for $\textbf{\textit{K}}$ with optimal bounds $B_2<A_2<0<A_1<B_1$ 
such that $f_n\neq~0~\forall~n\in{\mathbb{N}}$. If for every infinite increasing sequence $\{n_k\}\in{\mathbb{N}}$, $\{f_{n_k}\}$ 
is a frame sequence with optimal bounds $B_2<A_2<0<A_1<B_1$, then $\{f_n\}$ is an exact frame.
\end{theorem}
\begin{proof}
Suppose $\{f_n\}$ is not exact. Then there exists an $m\in{\mathbb{N}}$ such that $f_m^+\in~[f_i^+]_{i\neq~m}$ 
and $f_m^-\in~[f_i^-]_{i\neq~m}$. Let $\{n_k\}$ be an increasing sequence given by $n_k=k$, $k=~1,~2,~\ldots,m-1$ 
and $n_k=k+1$, $k =m,m+1,\ldots$. Since $\{f_{n_k}\}$ is a frame 
for $\textbf{\textit{K}}$ with bounds $B_2<A_2<0<A_1<B_1$,\\
\begin{equation*}
A_1\|f^+\|^2\leq\sum_{n\neq{m}}|\langle~f^+,f_n^+\rangle|^2\leq~B_1\|f^+\|^2
\end{equation*}
\begin{equation*}
\textmd{ and~~~~ }A_2\langle~f^-,f^-\rangle~\leq\sum_{n\neq{m}}|\langle~f^-,f_n^-\rangle|^2\leq~B_2\langle~f^-,f^-\rangle.
\end{equation*}
Therefore, by frame inequality for the frame $\{f_n\}$, $|\langle~f^+,f_m^+\rangle|^2=0~\forall~f^+\in~\textbf{\textit{K}}^+$.
This gives $f_m^+=0$.
Similarly we can have $f_m^-=0$ i.e.$~f_m=0$ which is a contradiction.\\
The proof is complete.
\end{proof}

\subsection{Frame Potential}
In 2001, John Benedetto and Matthew Fickus \cite{jm} developed a theoretical notion of Frame potential, 
which is analogue to the potential energy in Physics \cite{dkde}. The frame potential of a collection of vectors must be a scaler 
quantity derived from the inner products between the vectors. In this section we define the frame potential for a collection of 
unit vectors in Krein space $\textbf{\textit{K}}$, similar to what had been done for frames in Hilbert space theory.
\begin{definition}
Let $\textbf{\textit{K}}$ be a Krein space of dimension n. Let $\textbf{\textit{F}}=\{f_i\}_{i=1}^{k}$ be a collection of vectors in $\textbf{\textit{K}}$ s.t. $\|f_i\|_J=1~\forall~i=1,2,..$. The frame potential for $\textbf{\textit{F}}$ is the quantity
$P_{\textbf{\textit{F}}}=\sum_{i=1}^{k}\sum_{j=1}^{k}(|\langle~f_i^+,f_j^+\rangle|^2+|\langle~f_i^-,f_j^-\rangle|^2)$.
\end{definition}
\begin{lemma}
Let $\textbf{\textit{F}}=\{f_i\}_{i=1}^{k}$ be a tight frame of unit vectors in the Krein space $\textbf{\textit{K}}$ of dimension n. 
Then the frame potential $P_\textbf{\textit{F}}$ of $\textbf{\textit{F}}$ is $\frac{k^2}{n}$.
\end{lemma}
\begin{proof}
If $\{f_i\}_{i=1}$ is a tight frame of unit vectors on an n-dimensional Krein space $\textbf{\textit{K}}$, then the frame bound satisfies
\begin{equation*}
B_2=A_2=C=-\frac{k}{n}~~\textmd{and}~~B_1=A_1=D=\frac{k}{n}.
\end{equation*}
Hence,
\begin{equation*}
\begin{split}
\sum(\sum|\langle~f_i^+,f_j^+\rangle|^2+\sum|\langle~f_i^+,f_j^+\rangle|^2) & =\sum(\frac{k}{n}\langle~f_i^+,f_i^+\rangle-\frac{k}{n}\langle~f_i^-,f_i^-\rangle)\\
& =\frac{k^2}{n}.
\end{split}
\end{equation*}
\end{proof}
\begin{theorem}
The minimum value of the frame potential for a set of $k(\geq{n})$ unit vectors $\textbf{\textit{F}}=\{x_i\}$ in an n-dimensional Krein 
space is $\frac{k^2}{n}$. This minimum is attained exactly when the vectors form a tight frame for $\textbf{\textit{K}}$.
\end{theorem}
\begin{proof}
From the preceding lemma, we know that if $\textbf{\textit{F}}$ is a tight frame then its associated frame potential is 
$\frac{k^2}{n}$. It remains to be shown that $\frac{k^2}{n}$ is a lower bound on the set of achievable frame potentials 
among such collection $\textbf{\textit{F}}$, and that every collection attaining this lower bound is a tight frame.

		Let $S_1$ and $S_2$ be the frame operator for $\textbf{\textit{F}}=\{x_i\}$ and let $\lambda_1^+,\ldots~,\lambda_n^+$ and $\lambda_1^-,\ldots~,\lambda_n^-$ be the two sets of eigenvalues of $S_1$ and $S_2$ respectively, counting multiciplity.
		Given $\theta_1$ and $\theta_2$ the analysis operators of $\textbf{\textit{F}}$, the Grammian operator is $G_1=\theta_1\theta_1^\ast$ and $G_2=\theta_2\theta_2^\ast$. The matrix ${G_1}^2$ has diagonal entries $d_{ii}=\sum|\langle~x_i^+,x_j^+\rangle|^2$ and the matrix ${G_2}^2$ has diagonal entries $d_{jj}^2=\sum|\langle~x_j^-,x_i^-\rangle|^2$.\\
		So we see that the frame potential
\begin{equation*}
\begin{split}
		P_\textbf{\textit{F}} & =tr(G_1^2)+tr(G_2^2)\\
		& =tr(S_1^2)+tr(S_2^2)\\
		& =\sum{\lambda_{i}^{+}}^{2}+\sum{\lambda_{i}^{-}}^{2}\\
		& =\sum({\lambda_{i}^{+}}^{2}+{\lambda_{i}^{-}}^{2}).
\end{split}
\end{equation*}
		So we have to minimize the quantity $\sum({\lambda_{i}^{+}}^{2}+{\lambda_{i}^{-}}^{2})$ subject to the constraint $\sum(\lambda_{i}^{+}+\lambda_i^-)=k$ \\
		(The diagonal elements of the Grammian matrix $G_1$ and $G_2$ are squared norms $\|x_i^+\|^2$ and $\|x_i^-\|^2$ respectively)\\
By letting $x=(\lambda_1^+,\ldots~,\lambda_m^+,\lambda_1^-,\ldots~,\lambda_{k-m}^-)$ and
		$y=(1,\ldots~,1,\ldots~,1)$ we have,
		$k=|\langle~x,y\rangle|~\leq~\|x\|\|y\|=\sqrt{P_\textbf{\textit{F}}}\sqrt{n}$.
		
		where the equality holds iff x is in the span of y. Therefore, the frame potential $P_\textbf{\textit{F}}$ has minimum value $\frac{k^2}{n}$ which is attained iff the eigenvalues are all equal in absolute value i.e. $\lambda_i^+=\frac{k}{n}$ and $\lambda_i^-=\frac{k}{n}$. So $S_1=\frac{k}{n}I$ and $S_2=\frac{k}{n}I$, hence $\textbf{\textit{F}}$ is a tight frame for $\textbf{\textit{K}}$. $~~~~~\square$
\end{proof}
\begin{definition}
Let $\textbf{\textit{K}}$ be a Krein space of dimension n. Let $\textbf{\textit{F}}=\{f_i\}_{i=1}^\mathbb{N}$ be a collection of unit vectors in $\textbf{\textit{K}}$, then the sequence $\textbf{\textit{F}}$ is said to be a FF-critical sequence if the following equations are satisfied
\begin{equation*}
S_1(f_i^+)=\lambda_i^+{f_i^+}\textmd{ and~}S_2(f_i^-)=\lambda_i^-{f_i^-}
\end{equation*}
where $S_1$ and $S_2$ are bijective, self-adjoint and positive operator described before. Also $\lambda_i^+,~ \lambda_i^+\in{\mathbb{R}}$.
\end{definition}
\begin{theorem}
A finite sequence of unit vectors $\{f_n\}$ is FF-critical in a Krein space $\textbf{\textit{K}}$ iff both the sequences $\{f_n^+\}$ and $\{f_n^-\}$ may be partitioned into a collection of mutually orthogonal sequences, each of which is a tight frame for its span.

Furthermore, the partition to be chosen explicitly to be $\{E_{\lambda^{+}}\}$ and $\{E_{\lambda^{-}}\}$, where $E_{\lambda^{+}}=\{f_n^+:S_1f_n^+=\lambda^+f_n^+\}$ and $E_{\lambda^{-}}=\{f_n^+:S_2f_n^-=\lambda^-f_n^-\}$. Also the frame constant of $E_{\lambda^{+}}$ is $\lambda^+$ and that of $E_{\lambda^{-}}$ is $\lambda^-$ and the spans of the $\{E_{\lambda^{+}}\}$ and $\{E_{\lambda^{-}}\}$ are precisely the non-zero eigenspaces of $S_1$ and $S_2$ respectively.
\end{theorem}
\begin{proof}
$(\Rightarrow)$  Assuming $\{f_n\}$ is a FF-critical sequence, define $E_{\lambda^{+}}=\{f_n^+:S_1f_n^+=\lambda^+f_n^+\}$. Since each $f_n^+$ is in one of the $E_{\lambda^{+}}$ by assumption, then $\{E_{\lambda^{+}}\}$ is a partition of $\{f_n^+\}$. Similarly $\{E_{\lambda^{-}}\}$ is a partition of $\{f_n^-\}$. To show that the members of $\{E_{\lambda^{+}}\}$ are mutually orthogonal, take any $f_{n_1}^+$, $f_{n_2}^+$ in $E_{\lambda_{1}^{+}}$, $E_{\lambda_{2}^{+}}$, respectively. Since $S_1$ is self-adjoint, so $\lambda_{1}^{+}\neq{\lambda_{2}^{+}}$ implies $\langle{f_{n_1}^+},f_{n_2}^+\rangle=0$. By similar arguments we can show that the members of $\{E_{\lambda^{-}}\}$ are also mutually orthogonal.

To show that $E_{\lambda_{1}^{+}}$ is a finite tight frame for $span{E_{\lambda_{1}^{+}}}$ with frame constant $\lambda_{1}^{+}$, take any $y^+\in{span{E_{\lambda_{1}^{+}}}}$. Then $S_1y^+=\lambda_1^+y^+$, and so $\langle{y^+},f_{n}^+\rangle=0$ for any $f_{n}^+\neq{E_{\lambda_{1}^{+}}}$. Thus, letting $S_{\lambda_{1}^{+}}$ denote the frame operator for the Bessel sequence $E_{\lambda_{1}^{+}}$, we have
\begin{equation*}
\begin{split}
S_{\lambda_{1}^{+}}y^+ & =\sum_{f_{n}^+\in{E_{\lambda_{1}^{+}}}}\langle{y^+},f_{n}^+\rangle{f_{n}^+}\\
& =\sum_{n}\langle{y^+},f_{n}^+\rangle{f_{n}^+}\\
& =S_1y^+\\
& =\lambda_{1}^{+}y^+\\
\end{split}
\end{equation*}
By similar arguments we can also show that $E_{\lambda_{1}^{-}}$ is also a finite tight frame for the $span{E_{\lambda_{1}^{-}}}$ with frame constant $\lambda_{1}^{-}$.\\\\
$(\Leftarrow)$  Now assume that $\{f_n\}$ is a normalized sequence s.t. $\{f_n^+\}$ and $\{f_n^-\}$ both partitioned into a collection of mutually orthogonal sequences, each of which is a tight frame for its span. Denote the partitions by $\{E_{\gamma^+}\}_{\gamma^+\in{I^+}}$ and $\{E_{\gamma^-}\}_{\gamma^-\in{I^-}}$, where $I^+$ and $I^-$ are appropriate index set. Take any $f_{m}^+\in{E_{\gamma^+}}$, where $E_{\gamma^+}$ is a finite tight frame for its span with frame constant $\gamma^+$ and the frame operator ${S_{1}}_{\gamma^+}$. Then since $f_{m}^+\in{spanE_{\gamma^+}}$, and since $\langle{f_{m}^+},f_{n}^+\rangle=0$ for $f_{n}^+\notin{E_{\gamma^+}}$ by assumption, we have
\begin{equation*}
\begin{split}
\gamma^+f_m^+ & ={S_{1}}_{\gamma^+}f_m^+\\
& =\sum_{f_{n}^+\in{E_{\lambda_{1}^{+}}}}\langle{f_m^+},f_{n}^+\rangle{f_{n}^+}\\
& =\sum_{n}\langle{f_m^+},f_{n}^+\rangle{f_{n}^+}\\
& =S_1f_m^+.
\end{split}
\end{equation*}
Thus each $f_m^+$ is an eigenvalue for $S_1$.\\

Similarly, we can show that each $f_m^-$ is an eigenvector of $S_2$, and therefore $\{f_n\}$ is a FF-critical sequence. $~~~~~\square$
\end{proof}
{\bf Acknowledgement.} SMH and SK gratefully acknowledge the support of Aliah University, Kolkata for providing all the facilities when the manuscript was prepared. SK also acknowledges the financial support of CSIR, Govt. of India.

\bibliographystyle{amsplain}

\end{document}